\documentclass[reqno, 12pt, reqno]{amsart}

\usepackage{amsfonts, amsthm, amsmath, amssymb}

\usepackage{hyperref}

\usepackage[margin=1.5in]{geometry}

\usepackage[all]{xy}

\RequirePackage{mathrsfs} \let\mathcal\mathscr

\numberwithin{equation}{section}

\newtheorem{theorem}{Theorem}[section]
\newtheorem{lemma}[theorem]{Lemma}

\newtheorem{corollary}[theorem]{Corollary}

\theoremstyle{definition}
\newtheorem*{ack}{Acknowledgements}

\renewcommand{\phi}{\varphi}
\renewcommand{\rho}{\varrho}

\renewcommand{\AA}{\mathbb{A}}
\newcommand{\A}{\mathbf{A}}

\newcommand{\ZZ}{\mathbb{Z}}

\newcommand{\QQ}{\mathbb{Q}}

\newcommand{\GG}{\mathbb{G}}

\newcommand{\cP}{\mathcal{P}}

\newcommand{\Gal}{{\rm Gal}}

\renewcommand{\leq}{\leqslant}

\renewcommand{\geq}{\geqslant}

\newcommand{\fo}{\mathfrak{o}}
\newcommand{\fa}{\mathfrak{a}}

\newcommand{\fp}{\mathfrak{p}}

\DeclareMathOperator{\ord}{ord}

\newcommand{\Br}{{\rm Br}}

\newcommand{\knot}{\mathfrak{K}}

\newcommand{\bbQ}{{\mathbb Q}}
\newcommand{\bbR}{{\mathbb R}}

\newcommand{\bbZ}{{\mathbb Z}}

\newcommand{\Nloc}{N_{\mathrm{loc}}(B)}
\newcommand{\Nce}{N_{\mathrm{ce}}(B)}
\newcommand{\Nglob}{N_{\mathrm{glob}}(B)}
\newcommand{\Nlocp}{N_{\mathrm{loc},+}(B)}
\newcommand{\Nglobp}{N_{\mathrm{glob},+}(B)}

\title[{Failures of the Hasse norm principle}]{The proportion 
of  failures of the\\ Hasse norm principle}

\author{T.D. Browning}
\address{School of Mathematics\\
University of Bristol\\ Bristol\\ BS8 1TW\\UK}
\email{t.d.browning@bristol.ac.uk}

\author{R. Newton}
\address{Max Planck Institute For Mathematics\\ 
Vivatsgasse 7\\
53111 Bonn\\
Germany}
\email{rachel@mpim-bonn.mpg.de}

\subjclass[2010]{11G35 (11R37, 12G05, 14G05)}
\date{\today}

\begin{document}

\begin{abstract}
For any number field 
we calculate the exact proportion of rational numbers which are everywhere locally a norm but not globally a norm from the number field.
\end{abstract}

\maketitle


\thispagestyle{empty}

\section{Introduction}

Let $K/k$ be a finite extension of number fields with associated 
id\`ele groups  $J_K$ and $J_k$.
Let $N_{K/k}:J_K\to J_k$ denote the norm map and, in the usual manner, consider the multiplicative groups 
$K^*$ and $k^*$ as subgroups of $J_K$ and $J_k$, respectively. 
The {\em Hasse norm principle} is said to hold for $K/k$ if 
$$
k^*\cap N_{K/k}J_K=N_{K/k}K^*,
$$
where we view the group $N_{K/k}K^*$  of global norms as a subgroup of 
$k^*\cap N_{K/k}J_K$, with finite index. 
When $k=\QQ$ 
the main result of this paper gives the exact proportion of 
elements in $\QQ^*\cap N_{K/\QQ}J_K$  which do not belong to $N_{K/\QQ}K^*$.

The classical Hasse norm theorem states that an extension $K/k$ satisfies the Hasse norm principle when it is cyclic.
For non-cyclic extensions there can be counter-examples. The most elementary example is provided by 
the biquadratic extension $\QQ(\sqrt{13},\sqrt{17})/\QQ$, with  $5^2$ being represented everywhere locally by norms but  
 not  globally. 
Subsequently, a vast literature has emerged around the Hasse norm principle, resulting in a long list of explicit conditions under which one can conclude that a given extension $K/k$ satisfies the Hasse norm principle. 
For example,  the Hasse norm principle holds if:
\begin{itemize}
\item[---]
the degree $[K:k]$ is prime  (Bartels \cite{bartels-81});  or
\item[---]
the Galois group of $N/k$ is dihedral, where $N$ is the normal closure of $K$ over $k$
(Bartels \cite{bartels-81'});  or
\item[---]
$K/k$ is Galois and every 
Sylow subgroup of the Galois group is cyclic (Gurak \cite[Cor.~3.2]{gurak}).
\end{itemize}
The {\em knot group} $\mathfrak{K}(K/k)$ of the extension $K/k$ is defined to be 
the finite quotient group $(k^*\cap N_{K/k}J_K)/N_{K/k}K^*$
 and the {\em knot number} $i(K/k)$ is 
the cardinality of this  group.  The Hasse norm principle fails for $K/k$ if and only if $i(K/k)>1$.
The calculation of knot groups can be 
very difficult in practice.   Razar \cite{razar} has provided a general 
method for computing $i(K/k)$ via the construction of large abelian extensions of $K$. 
For bicyclic extensions the situation is simpler and we will establish the following result.

\begin{theorem}\label{thm:call}
Let $K/k$ be a Galois extension of number fields with Galois group $\bbZ/m\bbZ\times ~\bbZ/n\bbZ$. For a place $v$ of $k$, let $g_v$ be the number of places of $K$ above $v$. Let $g$ be the greatest common divisor of all the $g_v$.
Then $\mathfrak{K}(K/k)\cong \bbZ/g\bbZ$.
\end{theorem}

An alternative point of view comes from observing  that the failure of the Hasse norm principle for $K/k$ is equivalent to the failure of the {\em Hasse principle} for the affine variety \begin{equation}
V: \quad N_{K/k}(\Xi)=c,
\label{eq:eq}
\end{equation}
for an element $c\in k^*$. This equation defines a principal homogeneous space for the norm one $k$-torus $T=R^1_{K/k}\GG_m$ and its associated Tate--Shafarevich group is isomorphic to the knot group  $\mathfrak{K}(K/k)$.
Work of 
Sansuc \cite{sansuc}  implies  that the failure of the Hasse principle 
on any smooth projective model of $V$ is controlled by 
the {\em Brauer--Manin obstruction}.  Following Manin \cite{manin},
one can therefore construct 
the  {\em unramified Brauer group} of $V$, which is defined to be 
 $\Br(V^c)=H_{\text{\'et}}^2(V^c,\GG_m)$ 
for any  smooth compactification $V^c$ of $V$, together with 
a Brauer set $V^c(\A_k)^{\Br}\subseteq V^c(\A_k)$ which contains the 
$k$-rational points of $V^c$ and so  potentially    obstructs the existence of rational points

In principle,
both the knot group and the Brauer--Manin obstruction can  produce a decision algorithm 
for the existence of $k$-rational points on a particular  variety $V$ given by \eqref{eq:eq}.
Specialising to biquadratic extensions $K/k$, H\"urlimann \cite{hurl'} has constructed an explicit isomorphism $$
\psi: k^*/\prod_{i=1}^3N_{K_i/k} K_i^*
\overset{\sim}{\longrightarrow} \knot(K/k), \quad \psi(t)=t^2,
$$ 
where $K_1,K_2,K_3$ are the three quadratic subfields of $K$, which he has used to produce  criteria for the solubility of  \eqref{eq:eq}. More recently, 
Wei \cite[\S 2]{wei} has used the Brauer--Manin obstruction to do the same thing for biquadratic extensions $K/k$ with $k=\QQ$.

Returning to the general case, let $K/k$ be an extension of number fields 
with $i(K/k)>1$.  In practice it can be very challenging to determine whether or not 
a given element $c\in k^*\cap N_{K/k} J_K$ belongs to $N_{K/k} K^*$, 
or equivalently, whether  or not $c\in k^*$ produces a counter-example to the Hasse principle for the affine variety $V$ in \eqref{eq:eq}.
In order to compute the Brauer--Manin obstruction one needs to construct the unramified Brauer group. 
While it is relatively easy to calculate $\Br(V^c)/\Br(k)$ (cf.\  \cite{CTK}), 
finding explicit generators that lift to elements of $\Br(V^c)$  entails considerable strain.
Assuming that  $K/k$ is Galois, 
Wei \cite{wei} has transformed this problem into one that involves 
constructing a special abelian extension of $K$, but so far this plan has only been executed for biquadratic  (and certain cyclotomic) extensions of $\QQ$.

The primary aim of this paper is to measure the density of counter-examples to the Hasse principle for $V$ for a general number field $K/\QQ$.  
In doing so, we will succeed in proving asymptotic formulae for each of the counting functions
\begin{align*}
N_{\mathrm{loc}}(B)&=\#\left\{t\in \QQ^*\cap N_{K/\bbQ}J_K: H(t)\leq B \right\},\\
N_{\mathrm{glob}}(B)&=\#\left\{t\in  N_{K/\bbQ}K^*: H(t)\leq B \right\},\\
\Nce&=\Nloc-\Nglob,
\end{align*}
where $H(t)=\max\{|a|,|b|\}$ if $t=a/b$ for coprime integers $a,b$.
This will allow us to draw the following conclusion. 

\begin{theorem}\label{thm:prop}
For any number field $K/\QQ$ we have 
$$
\lim_{B\rightarrow \infty} \frac{\Nce}{\Nloc} = 1-\frac{1}{i(K/\QQ)}.
$$
\end{theorem}

The  proof of this result avoids having to construct the unramified Brauer group explicitly and instead uses the knot group in an essential way.  A consequence of our result
is  that a positive proportion of elements of $\QQ^*\cap N_{K/\bbQ}J_K$ 
give counter-examples to  the Hasse principle for $V$ unless, of course, the knot group 
is trivial. For 
Galois extensions $K/\QQ$ with Galois group  $\ZZ/p\ZZ\times \ZZ/p\ZZ$, for example, 
Theorem \ref{thm:call}  shows that
the proportion  is $1-\tfrac{1}{p}$ in Theorem \ref{thm:prop} provided 
that there is no place with local degree $p^2$.

It is interesting to compare Theorem \ref{thm:prop}  with recent work of Bhargava \cite{bhargava}, 
which  shows that a positive proportion of 
plane cubic curves over $\ZZ$  that
are everywhere  locally soluble actually  produce counter-examples to the Hasse principle. 
 In \cite[Conj.~6]{bhargava}, the precise proportion is conjectured to be $1-\tfrac{1}{3}$, something that bears comparison with our own discovery.

It is now time to reveal our asymptotic formulae for $\Nloc$ and $\Nglob$.
Let $\cP_K$ be the set of unramified rational primes $p$ for which 
$\gcd(f_1,\dots,f_r)=1$ whenever $(p)$ factorises in $\fo_K$ as $\fp_1\dots\fp_r$, for prime ideals $\fp_i$ with $\#\fo_K/\fp_i=p^{f_i}$.
According to
the 
\v{C}ebotarev
 density theorem 
(see \cite[\S 3.2]{serre}, for example)
this set has a natural density and we call this density $\delta_K$. 
The calculation in \cite[Cor.~13.6]{neukirch} shows that   $\delta_K\geq 1/[K:\QQ]$, with 
equality   if and only if
$K/\QQ$ is Galois.  
Next, we let $I_K$ denote the group of non-zero fractional ideals of $K$ and let $N_{K/\bbQ} I_K$ denote the subgroup of $\bbQ_{>0}$ consisting of norms of fractional ideals.
Two finite abelian  groups that play a key role in our analysis 
are 
\begin{equation}\label{eq:def-G}
G_{\mathrm{loc}}=\frac{N_{K/\bbQ}I_K}{\QQ_{>0}\cap N_{K/\bbQ}J_K} \quad \text{ and }\quad
G_{\mathrm{glob}}=\frac{N_{K/\bbQ}I_K}{\QQ_{>0}\cap N_{K/\bbQ}K^*}.
\end{equation}
In due course we will show that 
$\bbQ_{>0}\cap N_{K/\bbQ}J_K$ is a subgroup of  $ N_{K/\QQ}I_K$ 
and, furthermore, that 
$\#G_{\mathrm{glob}}=
\#G_{\mathrm{loc}}\cdot i(K/\QQ)$ (see  Corollary \ref{lem:dub} and Lemma~\ref{lem:ratio}, respectively).
Bearing these definitions in mind, 
we may now record the following result.

\begin{theorem}\label{thm:asymptotic}
For any number field $K/\QQ$
there exists a constant  $c(K)>0$ such that 
$$
\Nloc=\frac{1}{\#G_{\mathrm{loc}}}
 \frac{c(K) B^2}{(\log B)^{2(1-\delta_K)}} 
 \big(1+o(1)\big)
$$
and 
$$
\Nglob=\frac{1}{\#G_{\mathrm{glob}}}
 \frac{c(K) B^2}{(\log B)^{2(1-\delta_K)}} 
 \big(1+o(1)\big).
$$
\end{theorem}

Recalling that 
$\Nce=\Nloc-\Nglob$,
the statement of
Theorem~\ref{thm:prop} is an immediate consequence of this result.

Loughran \cite{dan} has developed some far-reaching conjectures about 
counting the number of varieties in a family which contain points everywhere locally.
Suppose that $K/\QQ$ has degree $n$. We may view the equation 
$N_{K/\QQ}(\Xi)=t$ as a family of varieties in $ \AA^{n+1}$ fibered over $\AA^1$. 
The fibre over $t=0$ 
is {\em non-split}
and the asymptotic formula for 
$\Nloc$ in 
Theorem~\ref{thm:asymptotic} is  consistent with the predictions in \cite[\S 1.3]{dan}.

Our analysis of  $\Nloc$ and $\Nglob$ is based on the 
Dirichlet series approach of   Odoni \cite{odoni}. 
Using properties of Artin $L$-functions, Odoni succeeded in getting an  asymptotic formula for the number of {\em positive integers} in an interval  which belong to $N_{K/\QQ}K^*$.   We will show that his argument can be  extended  to handle the  counting functions $\Nloc$ and $\Nglob$.

Returning once more to biquadratic extensions $K=\QQ(\sqrt{a},\sqrt{b})$ with 
knot number $i(K/\QQ)>1$, consider 
the $5$-dimensional
affine hypersurface
\begin{equation}\label{eq:coflasque}
(x^2-ay^2)(z^2-bt^2)(u^2-abw^2)=d,
\end{equation}
for $d\in \QQ^*$. A trivial modification of related work due to de la Bret\`eche and Browning \cite{coflasque} would guarantee counter-examples to the Hasse principle for this equation for a positive proportion of $d\in \QQ^*$, which is in contrast to the asymptotic formula for $\Nce$  provided by Theorem \ref{thm:asymptotic}.
Recent work of Colliot-Th\'el\`ene \cite[Props.~2.1 and~3.1]{ct-coflasque} implies that 
$c=d^2$  is a counter-example to the Hasse principle for $V$ whenever $d\in \QQ^*$ produces a counter-example to the Hasse principle for \eqref{eq:coflasque}. 
In fact, many of the known failures of  the Hasse principle for $V$ (see   \cite[Exercise 5]{CF}) involve the square of rational numbers and it is tempting to suppose that the set of rational numbers producing  counter-examples to the Hasse principle is very sparse.
Our work shows that  this set has density
$\tfrac{1}{2}$ in $\QQ^*\cap N_{K/\QQ}J_K$ and 
density  zero in  $\QQ^*$, 
but that it is not nearly  as sparse as the set of squares.

For biquadratic extensions $K/\QQ$ we have $\delta_K=1/4$ in Theorem \ref{thm:asymptotic} and  
it follows from Theorem \ref{thm:call} that $\knot(K/\QQ)\cong \ZZ/2\ZZ$ when $i(K/\QQ)>1$.  
In  the special case 
$K=\QQ(\sqrt{13},\sqrt{17})$ an inspection of \cite[Example~10]{wei} shows that  $\knot(K/\QQ)$ is  generated by $-1$.
This  easily implies that $\Nglob=\tfrac{1}{2}\Nloc$, an observation which affords the  
following especially  succinct result.

\begin{theorem}\label{thm:13-17}
Let $K=\QQ(\sqrt{13},\sqrt{17})$. Then there exists $c>0$ such that 
$$
\Nce
=\frac{1}{2}\Nloc
\sim  \frac{cB^2}{(\log B)^{3/2}}.
$$
\end{theorem}

We close the introduction by 
summarising the contents of the paper.
In \S \ref{s:norms} we shall 
establish Theorem \ref{thm:call} and 
discuss the group
$N_{K/\QQ}I_K$ of norms of fractional ideals belonging to $K$, in preparation for the proof of Theorem \ref{thm:asymptotic}. Finally in \S \ref{s:proof} we shall prove  this result. 

\begin{ack}
While working on this paper the 
first author
was  supported by {\em ERC grant} \texttt{306457}. 
The authors are grateful to Jean-Louis Colliot-Th\'el\`ene, 
Daniel Loughran and Damaris Schindler for comments on an earlier version of this paper.
\end{ack}

\section{Norm groups and class field theory}\label{s:norms}

We begin by proving Theorem \ref{thm:call}. 
Let 
$K/k$ be a Galois extension with Galois group $\bbZ/m\bbZ\times~\bbZ/n\bbZ$.
By H\"{u}rlimann \cite[Cor.~2.10]{hurl}, we have $\mathfrak{K}(K/k)\cong H^3(G, K^*)$ for a bicyclic extension. Let $C_K=J_K/K^*$ denote the id\`{e}le class group of $K$. We have  $H^3(G,J_K)=~1$ (see \cite[Cor.~8.1.3]{neukirch2}, for example). Taking Galois cohomology of the short exact sequence
\[1\rightarrow K^*\rightarrow J_K\rightarrow C_K\rightarrow 1\]
therefore gives
\begin{equation}
\label{BrauerGpSeq}
\dots \rightarrow H^2(G,J_K)\rightarrow H^2(G,C_K)\rightarrow H^3(G,K^*)\rightarrow 1.
\end{equation}
By \cite[Thm.~2.7]{Hochschild-Nakayama}, the image of $H^2(G,C_K)$ in $H^3(G,K^*)$ is a cyclic group of order $g$, where $g$ is as in the statement of Theorem \ref{thm:call}.
Since \eqref{BrauerGpSeq} is exact, this completes the proof of Theorem~\ref{thm:call}.

\medskip

Returning to general extensions $K/\QQ$, 
in preparation for our analysis of $\Nloc$ and $\Nglob$
we   collect  some facts about the  group 
$N_{K/\bbQ} I_K$ of norms of fractional ideals of $K$ and its subgroups.
The following result is well known.

\begin{lemma}
\label{idealnorms}
Let $\alpha\in\bbQ$. Then $\alpha\in N_{K/\QQ}I_K$  if and only if $\alpha>0$ and for every prime number $p$ the greatest common divisor of the residue degrees $f_{\mathfrak{p}/p}$ of the primes above $p$ divides $\ord_p(\alpha)$.
\end{lemma}

\begin{proof}
Let $\fa\in I_K$. Then 
\[N_{K/\bbQ}(\fa)=\prod_{p}{p^{\sum_{\mathfrak{p}\mid p}{f_{\mathfrak{p}/p}\ord_\mathfrak{p}(\fa)}}}\]
where the product runs over all prime numbers. In particular, the forward implication follows since 
$\ord_p(N_{K/\bbQ}(\fa))=\sum_{\mathfrak{p}\mid p}{f_{\mathfrak{p}/p}\ord_\mathfrak{p}(\fa)}$.  Now let $g_p$ be the greatest common divisor of the residue degrees of the primes above $p$. Write $g_p=\sum_{\mathfrak{p}\mid p}{a_{\mathfrak{p}}f_{\mathfrak{p}/p}}$ for $a_{\mathfrak{p}}\in\bbZ$. Then 
$$p^{g_p}=N_{K/\bbQ}\Bigl(\prod_{\mathfrak{p}\mid p}{\mathfrak{p}^{a_{\mathfrak{p}}}}\Bigr).$$
For $\alpha>0$ we have $\alpha=\prod_p{p^{\ord_p(\alpha)}}$, giving the  reverse implication.
\end{proof}

\begin{corollary}\label{lem:dub}
If $\alpha \in \bbQ_{>0} \cap N_{K/\bbQ}J_K$, then $\alpha\in N_{K/\QQ}I_K$.
\end{corollary}

\begin{proof}
Let $p$ be a prime number. Suppose that $\alpha =\prod_{\mathfrak{p}\mid p}{N_{K_\mathfrak{p}/\bbQ_p}(\beta_{\mathfrak{p}})}$ with $\beta_\mathfrak{p}\in K_\mathfrak{p}^*$. Then \[\ord_p(\alpha)=\sum_{\mathfrak{p}\mid p}{\ord_p(N_{K_{\mathfrak{p}}/\bbQ_p}(\beta_\mathfrak{p}))}
= \sum_{\mathfrak{p}\mid p}{f_{\mathfrak{p}/p}\ord_\mathfrak{p}(\beta_\mathfrak{p})}.\]
Therefore, the greatest common divisor of the residue degrees $f_{\mathfrak{p}/p}$ divides $\ord_p(\alpha)$. Now apply Lemma \ref{idealnorms}.
\end{proof}

 Following Odoni  \cite[\S 2]{odoni}, we observe that the norm map induces isomorphisms
\[\frac{I_K}{N_{K/\bbQ}^{-1}(\QQ_{>0}\cap N_{K/\bbQ}K^*)}\longrightarrow \frac{N_{K/\bbQ}I_K}{\QQ_{>0}\cap  N_{K/\bbQ}K^*}=G_{\mathrm{glob}}\]
and
\[\frac{I_K}{N_{K/\bbQ}^{-1}(\QQ_{>0}\cap N_{K/\bbQ}J_K)}\longrightarrow \frac{N_{K/\bbQ}I_K}{\QQ_{>0}\cap N_{K/\bbQ}J_K}=G_{\mathrm{loc}},\]
in the notation of \eqref{eq:def-G}.
Let 
$P_K^+=\{\alpha\fo_K: \sigma(\alpha)>0 \ \ \forall \ \sigma :K\hookrightarrow \bbR\}$
be the group of totally positive principal fractional ideals of $K$. 
Recall that the ideal-theoretic formulation of global class field theory gives an inclusion-reversing bijection between subextensions $L$ of the narrow class field of $K$ and subgroups $H$ of $I_K$ satisfying
\[I_K\supset H\supset P_K^+.\]
For a subextension $L$ corresponding to a subgroup $H$ as above, the Artin map gives an isomorphism
\[I_K/H\rightarrow \Gal(L/K).\]
We have 
\[I_K\supset N_{K/\bbQ}^{-1}(\QQ_{>0}\cap N_{K/\bbQ}J_K)\supset N_{K/\bbQ}^{-1}(\QQ_{>0}\cap N_{K/\bbQ}K^*)\supset P_K^+.\]
Therefore, class field theory tells us that both $N_{K/\bbQ}^{-1}(\QQ_{>0}\cap N_{K/\bbQ}J_K)$ and $N_{K/\bbQ}^{-1}(
\QQ_{>0}\cap N_{K/\bbQ}K^*)$ correspond to finite abelian subextensions of the narrow class field of $K$ which we will denote by $L_{\textrm{loc}}$ and $L_{\textrm{glob}}$ respectively. We have
\[\Gal(L_{\textrm{loc}}/K)\cong \frac{I_K}{N_{K/\bbQ}^{-1}(\QQ_{>0}\cap N_{K/\bbQ}J_K)}\cong G_{\mathrm{loc}}\]
and
\[\Gal(L_{\textrm{glob}}/K)\cong \frac{I_K}{N_{K/\bbQ}^{-1}(\QQ_{>0}\cap N_{K/\bbQ}K^*)}\cong 
G_{\mathrm{glob}}.\]
Observe that $G_{\mathrm{glob}}$ surjects onto $G_{\mathrm{loc}}$ with kernel 
\[\frac{\QQ_{>0}\cap N_{K/\bbQ}J_K}{\QQ_{>0}\cap N_{K/\bbQ}K^*}.\]
The following result shows that this kernel is equal to the knot group.

\begin{lemma}\label{lem:ratio}
We have
\[\frac{\QQ_{>0}\cap N_{K/\bbQ}J_K}{\QQ_{>0}\cap N_{K/\bbQ}K^*}=\mathfrak{K}(K/\bbQ).\]
In particular, 
${\#G_{\mathrm{glob}}}=
{\#G_{\mathrm{loc}}}\cdot\#\knot(K/\QQ).$
\end{lemma}
\begin{proof}
The second part of the lemma is obvious. To see the first part,
we have 
\[\frac{\QQ_{>0}\cap N_{K/\bbQ}J_K}{\QQ_{>0}\cap N_{K/\bbQ}K^*}\hookrightarrow \mathfrak{K}(K/\bbQ)\]
with cokernel 
\[\frac{\QQ^*\cap N_{K/\bbQ}J_K}{(\QQ_{>0}\cap N_{K/\bbQ}J_K)N_{K/\bbQ}K^*}.\]
We claim that this cokernel is trivial. If $K$ is totally complex then the claim follows since 
$\QQ^*\cap N_{K/\bbQ}J_K=\QQ_{>0}\cap N_{K/\bbQ}J_K$. 
Now suppose that $K$ has at least one real embedding.
 It suffices to show that there exists $\alpha\in K^*$ with $N_{K/\bbQ}(\alpha)<0$. Let us write $K=\bbQ(\beta)$ for some $\beta\in K$. Let $\sigma_1,\dots ,\sigma_r$ be the embeddings of $K$ into $\bbR$, ordered such that 
$\sigma_1(\beta)<\sigma_2(\beta)<\dots < \sigma_r(\beta).$ Let $m$ be the number of embeddings $\sigma_i$ for which $\sigma_i(\beta)<0$. If $m$ is odd, then $N_{K/\bbQ}(\beta)<0$. If $m$ is even, choose $\epsilon\in\bbQ$ such that adding $\epsilon$ to $\beta$ changes the sign of $\sigma_i(\beta)$ for exactly one embedding $\sigma_i$. Then $N_{K/\bbQ}(\beta +\epsilon)<0$.
\end{proof}

\section{Asymptotics for 
\texorpdfstring{$\Nloc$}{Nloc}
and 
\texorpdfstring{$\Nglob$}{Nglob}}\label{s:proof}

Our method for  estimating 
 $\Nloc$ and $\Nglob$ is based on  work of 
 Odoni  \cite[\S 1 and \S 2]{odoni}.  
 Recall that $I_K$ is the group of non-zero fractional ideals of $K$.
 We shall proceed by proving 
 asymptotic formulae for analogues of $\Nloc$ and $\Nglob$ in which the  elements to be counted are  
 restricted to finite index subgroups of  $N_{K/\QQ}I_K$.
The characters of the finite quotient groups are used to write down explicit indicator functions which allow us to phrase the counting problems in terms of Dirichlet series. By realising the finite quotient groups in question as Galois groups of abelian subextensions of the narrow class field of $K$, we can apply Odoni's method \cite[\S 2]{odoni} to relate each Dirichlet series to an Artin $L$-series in order to estimate the restricted counting functions asymptotically.
Finally, we shall show how 
 these counts can be related to 
  $\Nloc$ and $\Nglob$ in order to  establish Theorem \ref{thm:asymptotic}.

All of our  indicator functions will involve the
indicator function $\delta$ on $\bbQ_{>0}$ given by
 $$\delta(x)= \begin{cases}1, & \textrm{ if } x\in N_{K/\bbQ} I_K,\\
0, &\textrm{ otherwise.}\end{cases}
$$
This is a multiplicative function by \cite[Lemma 1.1]{odoni}.

\subsection{Local norms}
\label{sec:localnorms}

The aim of this section is to provide an asymptotic formula for $\Nloc$. 
We begin by analysing the counting function
$$
N_{\mathrm{loc},+}(B)=\#\left\{t\in \QQ_{>0}\cap N_{K/\bbQ}J_K: H(t)\leq B \right\}.
$$
Using the results of \S \ref{s:norms}, together with   the  multiplicativity of $\delta$, we find that (cf. \cite[Eq.~(2.3)]{odoni})
\[N_{\textrm{loc},+}(B)=
\frac{1}{\#G_{\mathrm{loc}}}\sum_{\chi}\sum_{a \leq B}{\delta(a)\chi(a)
\sum_{\substack{b\leq B\\ \gcd(a,b)=1}}
{\delta(b)\chi(b)^{-1}}},
\]
where $\chi$ runs through all the characters of the finite group $G_{\textrm{loc}}$.
Let us first consider the contribution  from the trivial character $\chi=\chi_0$. 
Taking $k=a$  
in \cite[Thm.~IIA]{odoni}
we deduce that  there is a constant $c_1(K)>0$ such that 
$$
\sum_{\substack{b\leq B\\ \gcd(a,b)=1}}{\delta(b)}
=
 \frac{c_1(K)\lambda(a) B}{(\log B)^{1-\delta_K}} \left(1+O\left(\frac{1}{\log B}\right)\right)
$$
uniformly in $a$, where 
$$\lambda(a)=\prod_{p\mid a}\left(1+\frac{\delta(p)}{p}+
\frac{\delta(p^2)}{p^2}+\dots\right)^{-1}.
$$
(See  \cite[Eq.~(1.9)]{odoni} for confirmation that the exponent of $\log B$
agrees with what is written here.)
The contribution to $\Nlocp$ from the trivial character is therefore found to be 
$$\frac{1}{\#G_{\mathrm{loc}}}
 \frac{c_1(K) B}{(\log B)^{1-\delta_K}} 
\sum_{a\leq B}{\delta(a)\lambda(a)}
 \left(1+O\left(\frac{1}{\log B}\right)\right).
$$
The function $\lambda(a)$ has constant average order of magnitude and 
an inspection of the proof of \cite[Thm.~IIA]{odoni} confirms that 
\begin{equation}\label{eq:with lambda}
\sum_{a\leq B}{\delta(a)\lambda(a)}=
 \frac{c_2(K)B}{(\log B)^{1-\delta_K}} \left(1+O\left(\frac{1}{\log B}\right)\right),
\end{equation}
for an appropriate constant $c_2(K)>0$.
Hence the trivial character makes the overall contribution
$$\frac{1}{\#G_{\mathrm{loc}}}
 \frac{c(K) B^2}{(\log B)^{2(1-\delta_K)}} 
 \left(1+O\left(\frac{1}{\log B}\right)\right)
 $$
to $\Nlocp$,
where $c(K)=c_1(K)c_2(K)$.

Turning to the contribution from the non-trivial characters, an application of the triangle inequality shows that they make an overall contribution 
\begin{equation}\label{eq:non-trivialcontribution}
\leq \frac{1}{\#G_{\mathrm{loc}}}
\sum_{\chi\neq \chi_0}
\sum_{a\leq B}{\delta(a)} \Bigg|
\sum_{\substack{b\leq B\\ \gcd(a,b)=1}}{\delta(b)\chi(b)^{-1}}\Bigg|.
\end{equation}
We now  appeal to the proof of \cite[Thm.~IIB]{odoni}, which shows that for a non-trivial character $\chi$ of $G_{\mathrm{loc}}$
there exists a constant $e_K>0$ such that 
\begin{equation}\label{eq:boundnon-trivial}
\Bigg|\sum_{\substack{b\leq B\\ \gcd(a,b)=1}}{\delta(b)\chi(b)^{-1}} \Bigg|=O\left( \frac{\lambda(a) B}{(\log B)^{1-\delta_K+e_K}}\right)
\end{equation}
uniformly in $a\leq B$.  This part of Odoni's proof involves relating the Dirichlet series 
\[\sum_{b=1}^{\infty}{\delta(b)\chi(b)b^{-s}}\]
to an Artin $L$-series, using the fact that $\chi$ is a character of the Galois group of an abelian subextension of the narrow class field of $K$, as shown in \S \ref{s:norms}.

Combining \eqref{eq:non-trivialcontribution} and \eqref{eq:boundnon-trivial} with \eqref{eq:with lambda}, we conclude that
\begin{equation}\label{eq:nlocp}
\Nlocp=\frac{1}{\#G_{\mathrm{loc}}}
 \frac{c(K) B^2}{(\log B)^{2(1-\delta_K)}} 
 \big(1+o(1)\big).
\end{equation}

It remains to relate $\Nlocp$ to $\Nloc$ in order to prove the  first part of Theorem \ref{thm:asymptotic}.
If $K$ is totally complex then 
$\Nloc=\Nlocp$. Now suppose that $K$ has at least one real embedding. If $-1\in N_{K/\QQ}J_K$ then
$\Nloc=2\Nlocp$. The final case to consider is therefore the case in which $K$ has at least one real embedding and $-1\notin N_{K/\QQ}J_K$.

We shall handle $\Nloc$ via the observation that the map $t\mapsto |t|$ gives a bijection of sets
$$
\left\{t\in \QQ\cap N_{K/\bbQ}J_K: H(t)\leq B \right\}
\rightarrow 
\left\{t\in T_K: H(t)\leq B \right\},
$$
where $T_K=\QQ_{>0}\cap [-N_{K/\QQ} J_K \cup N_{K/\QQ} J_K]$.
Thus 
$$
\Nloc=\#
\left\{t\in T_K: H(t)\leq B \right\}.
$$
In the proof of Lemma \ref{lem:ratio}, we saw that when $K$ has at least one real embedding $N_{K/\QQ}K^*$ contains negative elements. When combined with our assumption that $-1\notin N_{K/\QQ}J_K $, this shows that $T_K$ contains $\QQ_{>0}\cap N_{K/\QQ} J_K$ as an index $2$ subgroup. 
In particular, it follows that  
$
G=N_{K/\QQ} I_K/T_K$
is a finite abelian group, which by class field theory corresponds to an abelian subextension of 
the narrow class field of $K$. It is clear that 
$$
\#G=\frac{1}{2}\#G_{\mathrm{loc}}, 
$$
whence
$$
\Nloc
=\frac{2}{\#G_{\mathrm{loc}}}\sum_{\chi}\sum_{a\leq B}\delta(a)\chi(a)\sum_{\substack{b\leq B\\ \gcd(a,b)=1}}\delta(b)\chi(b)^{-1},
$$
where 
 $\chi$ runs through all the characters of $G$. 
This can now be estimated using 
Odoni's argument.  
Once combined with \eqref{eq:nlocp},
we conclude that 
$$
\Nloc=\frac{2^{j(K)}}{\#G_{\mathrm{loc}}}
 \frac{c(K) B^2}{(\log B)^{2(1-\delta_K)}} 
 \big(1+o(1)\big),
$$
where 
\begin{equation}\label{eq:j}
j(K)=\begin{cases}
1, & \text{if $K$ has a real embedding,}\\
0, & \text{if $K$ is totally complex}.
\end{cases}
\end{equation}
This is satisfactory for the first part of Theorem \ref{thm:asymptotic}.

\subsection{Global norms}

Our starting point is an asymptotic formula for 
$$
N_{\mathrm{glob},+}(B)=\#\left\{t\in \QQ_{>0}\cap N_{K/\bbQ}K^*: H(t)\leq B \right\}.
$$
As before, following \cite[\S 2]{odoni}, we find that 
\begin{align*}
N_{\textrm{glob},+}(B)
&=\frac{1}{\#G_{\textrm{glob}}}\sum_{\chi}\sum_{a\leq B}{\delta(a)\chi(a)\sum_{\substack{b\leq B\\ \gcd(a,b)=1}}
{\delta(b)\chi(b)^{-1}}}
\end{align*}
where now $\chi$ runs through all the characters of the finite group $G_{\textrm{glob}}$. The same argument used to 
establish \eqref{eq:nlocp} runs through as before with the outcome that 
$$
\Nglobp=\frac{1}{\#G_{\mathrm{glob}}}
 \frac{c(K) B^2}{(\log B)^{2(1-\delta_K)}} 
  \big(1+o(1)\big).
$$
Here the values of the constants $c(K)$ and $\delta_K$ are exactly as in \eqref{eq:nlocp}.
We would like to deduce from this that 
\begin{equation}\label{eq:331}
\Nglob=\frac{2^{j(K)}}{\#G_{\mathrm{glob}}}
 \frac{c(K) B^2}{(\log B)^{2(1-\delta_K)}} 
  \big(1+o(1)\big),
\end{equation}
where $j(K)$ is given by \eqref{eq:j}. This will suffice for  Theorem~\ref{thm:asymptotic}. 

If $K$ is totally complex then it is clear that
$\Nglob=\Nglobp$ and so we are done. Suppose  that $K$ has at least one real embedding. 
If  $-1\in N_{K/\QQ}K^*$ then 
$\Nglob=2\Nglobp$, which is satisfactory.
The remaining case to consider is therefore the case in which $K$ has at least one real embedding and $-1\notin N_{K/\QQ}K^*$. This is done in the same way as in \S\ref{sec:localnorms}, using the bijection
$$
\left\{t\in N_{K/\bbQ}K^*: H(t)\leq B \right\}
\rightarrow 
\left\{t\in U_K: H(t)\leq B \right\}, \quad t\mapsto |t|,
$$
where $U_K=\QQ_{>0}\cap [-N_{K/\QQ} K^* \cup N_{K/\QQ} K^*]$.

\end{document}